\documentclass{amsart}
\usepackage{graphicx,amsmath,mathtools,amssymb,epsfig,color}
\usepackage{hyperref,epstopdf,wrapfig,marvosym}
\usepackage[utf8]{inputenc} 
\usepackage[left]{lineno}
\usepackage[paperwidth = 8.5in, paperheight = 11in,total={6.5in,8.5in}]{geometry} 
\newtheorem{theorem}{Theorem}[section]
\newtheorem{lemma}[theorem]{Lemma}
\newtheorem{definition}[theorem]{Definition}
\newtheorem{proposition}[theorem]{Proposition}
\newtheorem{remark}[theorem]{Remark}

\newtheorem{conjecture}[theorem]{Conjecture}
\numberwithin{equation}{section}
\newtheorem{example}[theorem]{Example}
\def \hfillx {\hspace*{-\textwidth} \hfill}

\hypersetup{pdfstartview={XYZ null null 1.00}}

\begin{document}
	
	\title{On the rack homology of graphic quandles}
	
	\author{Sujoy Mukherjee}
	\address{Department of Mathematics, The George Washington University, USA.}
	\email{sujoymukherjee@gwu.edu}
	
	\author{J\'{o}zef H. Przytycki}
	\address{Department of Mathematics, The George Washington University, USA \and University of Gda\'{n}sk, Poland.}
	\email{przytyck@gwu.edu}
	
	
	\date{September 08, 2017, and in revised form, December 05, 2017.}
	
	
	\begin{abstract}
		
		{\bf This paper has partially a novel and partially a survey character.  We start with a short review of rack (two term) homology of self distributive algebraic structures (shelves) and their connections to knot theory. We concentrate on a sub-family of quandles satisfying the graphic axiom. For a large family of graphic quandles (including infinite ones), we compute the second rack homology groups. Finally, we propose conjectures based on our computational data.}
		
	\end{abstract}

	\maketitle
	
	
	Quandles are algebraic structures with axioms motivated by the Reidemeister moves from knot theory (\cite{Joy,Mat}). More general algebraic structures such as racks, spindles, and shelves are obtained by discarding some of the axioms of a quandle. A number of homology theories related to quandles have been developed in the last three decades with connections to knot theory. Rack (two term) homology was introduced in \cite{FRS1,FRS2,FRS3} and was later enhanced into quandle homology in \cite{CJKLS} to define quandle cocycle invariants for classical knots and links. While for rack homology it is enough to work with shelves, for quandle homology, spindles are necessary (see Section \ref{intro}).
	
	\begin{figure}[ht]
		\centering
		\includegraphics[scale=0.5]{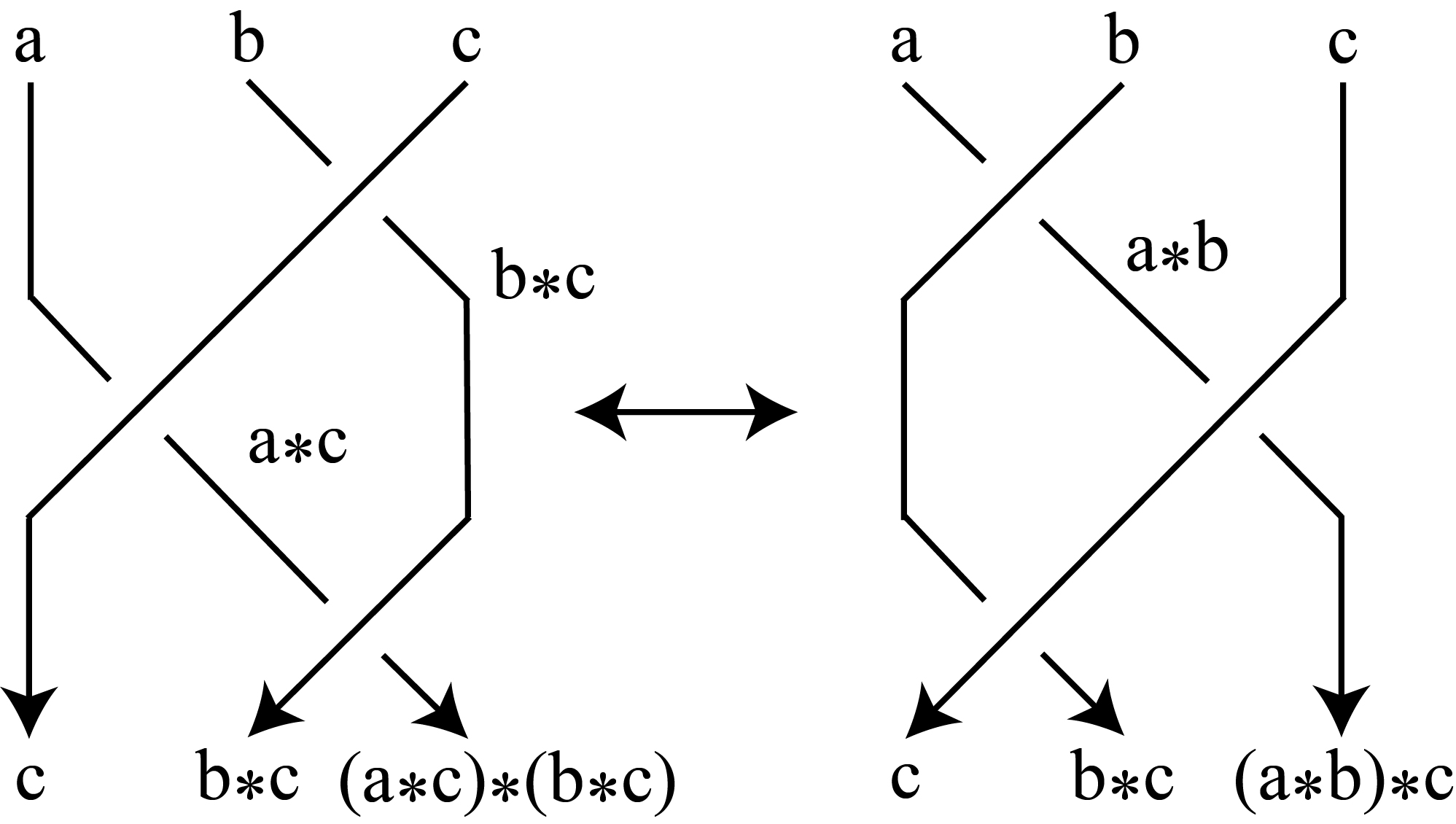}
		\caption{The third Reidemeister move and the self distributivity axiom}
	\end{figure}
	
	Rack and quandle homology theories have been studied extensively. The free part of rack and quandle homology for finite racks and quandles is completely determined in \cite{EG, LN}. The torsion part of rack and quandle homology have been studied in \cite{EG, LN, Nos2, NP1,NP2,NP3,PY}.
		
	The paper is organized as follows. In the first section, we introduce the basic notions related to self distributive algebraic structures. We also introduce the notion of graphic quandles and discuss our main example. In Section \ref{section 2}, after a brief outline of rack and quandle homology, we compute the torsion subgroups of the second rack homology groups of some of the graphic quandles from our main example. Additionally, we compute the second homology for some infinite graphic quandles. We describe quandle cocycle extensions of chosen graphic quandles. Finally, we finish with some open problems and computational data.
	
\section{Introduction} \label{intro}
	
	We start with basic definitions and examples. A {\color{blue}{\bf shelf}} or a self distributive algebraic structure\footnote{In this paper, we use the right self distributivity axiom for shelves, unless otherwise stated. The main reason for this convention is that knot theory is the origin of the notion of a quandle.} is a magma $(X,*)$ such that for all $a,b,c \in X,$ $$(a*b)*c = (a*c)*(b*c).$$ A {\color{blue}{\bf rack}} is a shelf such that there exists $\bar{*}:X \times X \longrightarrow X$ and for all $a,b \in X$, $$(a\ \bar{*} \  b)*b = a = (a*b) \ \bar{*} \ b.$$ A {\color{blue} {\bf quandle}} is an idempotent rack, that is, $a*a = a,$ for all $a \in X.$ A {\color{blue} {\bf spindle}} is an idempotent shelf.
	
	Homomorphisms and isomorphisms between shelves are defined in the usual way. Let $\mathcal{O}$ denote the set of orbits in a shelf with respect to right multiplication. If $|\mathcal{O}| = 1,$ then the shelf is said to be connected. Following are some examples of the algebraic structures defined above.
	
	\begin{example}
		\
		\begin{enumerate}
			\item{Let $(G,\cdot)$ be a group. Then $(G,*)$ is a quandle with $a*b = b^{-n} \cdot a \cdot b^n,$ for all $a,b \in G.$ When $n=1,$ these quandles are known as {\color{blue}{\bf conjugation quandles}}.}
			\item{Let $M$ be a module over the ring $\mathbb{Z}[t^{\pm 1}].$ It forms a quandle with $a*b = ta + (1-t)b,$ for all $a,b \in M.$ These quandles are called {\color{blue}{\bf Alexander}} or {\color{blue} {\bf affine quandles}}.}
			\item {Let Let $(G,\cdot)$ be a group. Then $(G,*)$ is a quandle with $a*b = b \cdot a^{-1} \cdot b,$ for all $a,b \in G.$ It is called the {\color{blue}{\bf core quandle}} of the group $G$. When $G$ is an abelian group it is called a {\color{blue}{\bf Takasaki quandle}} while if $G$ is $\mathbb{Z}_n,$ it is called a {\color{blue}{\bf dihedral quandle}}. When $G$ is Abelian, we write $a*b = 2b - a,$ for $a,b \in G.$ }
		\end{enumerate}
	\end{example}
	
	In the past, several sub-families of racks and quandles have received extensive attention due to their extra structure. Here are a few of them.
	
	\begin{example}
		\
		\begin{enumerate}
			\item{A {\color{blue}{\bf kei}} (introduced by M. Takasaki in 1942) or an involutory quandle, in addition to the first and third axioms of a quandle satisfies a stronger variant of the second axiom: $* = \bar{*},$ that is, for all $a,b \in X,$ $(a*b)*b = a.$}
			\item{Let $(X,*)$ be a quandle. If for all $a,b \in X,$ the equation $a*x = b$ has a unique solution, then $(X,*)$ is called a {\color{blue} {\bf quasigroup quandle}}.}
			\item {A quandle $(X,*)$ is said to be {\color{blue} {\bf entropic}} or medial, if for all $a,b,c,d \in X,$ $(a*b)*(c*d) = (a*c)*(b*d).$}
	\end{enumerate}
	\end{example}

	We continue with our main example of a sub-family of quandles in the next subsection.
	
	\subsection{Graphic quandles}\label{1.1}
	
	The graphic axiom $a*b = (a*b)*a,$ was introduced by F. W. Lawvere in 1987 when studying graphic monoids \cite{Law}. Graphic monoids are identical to unital {\it left} shelves.\footnote{In \cite{CMP}, since the {\it right} self distributive axiom was used, the left version ($a*b = b*(a*b)$) of the graphic axiom often appears. The left graphic axiom does not behave well in racks. Fortunately, this is not the case with the right version of the graphic axiom!}
	
	\begin{definition}
		Let $(X,*)$ be a quandle. If for all $a,b \in X,$ $a*b = (a*b)*a,$ then we call $(X,*)$ a {\color{blue} {\bf graphic quandle}}.
	\end{definition}

	\begin{table}[h]\label{a}
		
		\centering
		\caption{Graphic quandles versus quandles}
		\begin{tabular}{|c|c|c|c|c|c|c|}
			\hline
			$n$&1&2&3&4&5&6\\
			\hline
			\# graphic quandles of size $n$&1&1&2&5&15&56\\
			\hline
			\# quandles of size $n$&1&1&3&7&22&73\\
			\hline
		\end{tabular}
	\end{table}
	
	The additional axiom in the above definition is called the graphic axiom. Graphic racks and graphic shelves are defined analogously. The graphic axiom is satisfied by many finite quandles (Table \ref{a}). Following is one of the ways to construct a large family of graphic magmas, in particular graphic quandles.
	
	\begin{example}\label{graphic_quandle_example}
		Let $\{X_i\}_{i \in \Lambda}$ be non-empty sets, $ X = \bigsqcup_{\{i \in \Lambda\}}X_i,$ and $f_{i,j} : X_i \longrightarrow X_i$ for all $i,j \in \Lambda$. Further let $*:X \times X \longrightarrow X$ be defined as follows. For $x_i \in X_i$ and $x_j \in X_j$, $x_i*x_j = f_{i,j}(x_i)$.
		
		\begin{enumerate}
			\item{If $f_{i,j}f_{i,k} = f_{i,k}f_{i,j}$ for all $i,j \in \Lambda$, then $(X,*)$ is a shelf. In other words, for fixed $i$, the functions $f_{i,j}$ commute.}
			\item{$(X,*)$ is a rack if in addition to being a shelf, all the functions $f_{i,j}$ for all $i,j \in \Lambda$ are bijections.}
			\item{For the idempotency axiom, we need $f_{i,i} = Id$.\footnote{$f_{i,i} = Id$ also forces the magma to be left self-distributive. Therefore, quandles constructed in this way are idempotent left self distributive magmas as well.}}
			\item{For the graphic axiom to be satisfied we need $f_{i,i} = Id$ as well.}
			\item {The requirement for the entropic axiom is same as that for self-distributivity. Therefore, all shelves obtained by this construction are entropic.\footnote{In \cite{JPSZ}, it is proven that a quandle is entropic if and only if it the homomorphic image of a quasi-affine quandle.}}
			\item{For the associativity axiom, the condition necessary is $f_{i,j}=f_{i,k}f_{i,j},$ for $i,j,k \in \Lambda.$ One way to ensure this is as follows: For $i \in \Lambda,$ let $f_{i,j} = f_{i,k},$ for all $j,k \in \Lambda$ with all $f_{i,j}$(s) as idempotent maps.}
			
		\end{enumerate}
		
		   Therefore, to obtain a quandle, we need a rack with all the functions $f_{i,i}$ to be the identity map.
		    However, to obtain a spindle it is enough for the shelf $(X,*)$ to have $f_{i,i} = Id$. Associative shelves are obtained when all the maps are idempotent, they commute pairwise and for given $i \in \Lambda,$ $f_{i,j} = f_{i,k},$ for all $j,k \in \Lambda.$ See \cite{CMP,Muk} for a more detailed treatment of associative shelves and their behavior in the self-distributive category.
	\end{example}

	\begin{table}[h]
		\caption{Graphic quandles with two (on the left) and three orbits.}
		\label{examples of GQ}
		\begin{minipage}{0.5\textwidth}
			\centering
			\begin{tabular}{c || c c c | c c c}
				$*$&0&1&2&3&4&5 \\
				\hline \hline
				0&0&0&0&1&1&1\\
				1&1&1&1&2&2&2\\
				2&2&2&2&0&0&0\\
				\hline
				3&4&4&4&3&3&3\\	
				4&5&5&5&4&4&4\\
				5&3&3&3&5&5&5\\
			\end{tabular}
		\end{minipage}
		\hfillx
		\begin{minipage}{0.5\textwidth}
			\centering
			\begin{tabular}{c || c c | c | c c c}
				$*$&0&1&2&3&4&5 \\
				\hline \hline
				0&0&0&0&1&1&1\\
				1&1&1&1&0&0&0\\
				\hline
				2&2&2&2&2&2&2\\
				\hline
				3&4&4&4&3&3&3\\	
				4&5&5&5&4&4&4\\
				5&3&3&3&5&5&5\\
			\end{tabular}
		\end{minipage}
	\end{table}

	Table \ref{examples of GQ} shows two examples of graphic quandles constructed using the previous construction. The one on the left has two orbits. The one on the right has three orbits. Many among the finite graphic quandles can be constructed by this method. For some of these graphic quandles, we introduce special notation to make it convenient later in the paper. Let $GQ(o_1 \ | \ o_2 \ | \ \cdots \ | \ o_k)$ denote the graphic quandle with $k$ orbits $O_1,O_2,...,O_k$ of size $o_1,o_2,...,o_k$ respectively. We assume also that functions $f_{i,j} : O_i \longrightarrow O_i$ are $o_i$-cycles (not depending on $j$ and denoted by $f_i$) from the permutation group $S_{o_i}$ for all $0 \leq i \neq j \leq k.$ Observe that for $GQ(o_1 \ | \ o_2 \ | \ \cdots \ | \ o_k),$ $X_i = O_i,$ that is, the sets $X_i$ are equal to the orbits $O_i.$

	In the following remark, we compare the construction of the previous example with the already known sub-families of quandles.

	\begin{remark}
		Let $(X,*)$ be a graphic quandle constructed as in Example \ref{graphic_quandle_example}. Further, let $a \in X_i$ and $b \in X_j.$
		\begin{enumerate}
			\item{$(( \cdots (a*b)*b) \cdots *b) = a*^nb = f_{i,j}^n(a).$ Therefore, if $f_{i,j}$ has order $n,$ $(X,*)$ is an $n$-quandle. In particular, when $n=2,$ then $(X,*)$ is a kei.}
			\item{In an Alexander quandle defined for a module $M$ over $\mathbb{Z}[t^{\pm 1}],$ $(a*b)*a = (ta + (1-t)b)*a = t^2a + t(1-t)b + (1-t)a = (t^2 - t + 1)a + (t-t^2)b.$ Therefore, for the graphic axiom to hold, we need $ (t^2 - t + 1)a = ta,$ that is, $(t-1)^2\cdot M = 0.$ In particular, the Alexander quandle $\mathbb{Z}[t]/<(t-1)^2>$ is graphic. There are two interesting sub-cases:
				\begin{enumerate}
					\item{ Let $k \neq 0.$ Consider the module: $\mathbb{Z}[t]/<(t-1)^2, t - (k+1)>.$ We obtain in this case $Z_{k^2}$ with $i*j= (k+1)i -kj$.When $k = \pm 2,$ we obtain $R_4,$ the dihedral quandle of four elements.}
					\item {Consider the module: $\mathbb{Z}[t]/<3,(t-1)^2> = \mathbb{Z}_3[t]/<t^2 + t + 1>.$ This quandle is related to the 3-fold branched cover of $S^3$ branched along a link.}
				\end{enumerate}}
			\item{Quasigroup quandles are connected, but the quandles constructed in \ref{graphic_quandle_example} are not connected for $|X| > 1$.}
			\end{enumerate}
	\end{remark}

	Distributive sets of binary operations were defined in \cite{Deh, Prz, RS}. The graphic shelves in Example \ref{graphic_quandle_example} can be generalized to graphic multishelves. 
	
	\begin{proposition} Let $(X,*_f)$ and $(X,*_g))$ be shelves from Example \ref{graphic_quandle_example}. Then they form a distributive set if and only if functions $f_{i,j}$ and $g_{i,k}$ commute, that is, $f_{i,j}g_{i,k}=g_{i,k}f_{i,j},$ for all $i,j,k.$
	\end{proposition}
	\begin{proof}
		It suffices to check self distributivity: $(a*_fb)*_gc=(a*_gc)*_f(b*_gc)$ (\cite{Prz}). Let $a\in X_i,b\in X_j$ and $c\in X_k$, then we have:
		$$(a*_fb)*_gc= f_{i,j}(a)*_gc= g_{i,k}f_{i,j}(a),$$
		$$(a*_gc)*_f(b*_gc)= g_{i,k}(a)*_fg_{j,k}(b)= f_{i,j}g_{i,k}(a).$$
	\end{proof}

	The importance of forming distributive sets is the possibility of defining multiterm distributive homology by forming linear combinations of $\partial_n^{(*)}$ for operations in a distributive set \cite{Prz}. See also \cite{CPP}.

	Observe that in a spindle (i.e. an idempotent shelf), $(a*b)*a = (a*a)*(b*a) = a*(b*a).$ Therefore, in a {\bf graphic spindle}, for any arbitrary pair of elements $a*(b*a) = (a*b)*a = a*b.$ Further, in a graphic quandle $((a*b)*c)*a = ((a*b)*a)*(c*a) =(a*b)*(c*a).$\footnote{If a quasigroup satisfies the equality $((a*b)*c)*a = (a*b)*(c*a),$ then it is a group (\cite{PV}).}

	\begin{proposition}\label{smallp}
		\
		\begin{enumerate}
			\item {Let $(X,*,\bar{*})$ be a graphic magma satisfying for $a,b \in X,$ $(a*b)\ \bar{*} \ b = a = (a\ \bar{*} \ b)*b$ . Then, for $a \in X$, $a = a*a.$ Therefore, all such $(X,*, \ \bar{*})$ which are shelves are graphic quandles.}
			\item {Graphic spindles do not contain any non-trivial quasigroup subspindles.}
		\end{enumerate}
	\end{proposition}

	\begin{proof}\
		\begin{enumerate}
			\item {Let $a \in X.$ Then, $a*a = (a*a)*a$ by the graphic axiom, so that $(a*a)\ \bar{*} \ a = ((a*a)*a) \ \bar{*} \ a,$ which gives $a = a*a.$}
			\item {Let $Y \subset X$ be a quasigroup subquandle and $a \neq b \in Y.$ Then, $a*b = (a*b)*a \implies a*b = a*(b*a) \implies b = b*a,$ as $Y$ is a quasigroup. But $b = b*b \implies b*b = b*a \implies b = a,$ as $Y$ is a quasigroup which is a contradiction.}
		\end{enumerate}
	\end{proof}

	By Proposition \ref{smallp}, graphic quandles cannot contain a subquandle which is a quasigroup.  In fact, almost all finite quandles up to order six ($100$ out of the $107$) are either graphic quandles or contain a non-trivial quasigroup. Further, all graphic quandles up to order six can be constructed using Example \ref{graphic_quandle_example}. 
	
	The sets $X_i$ in Example \ref{graphic_quandle_example} can be infinite as in the following example.
	
	\begin{example}
		Let $f: \mathbb{Z} \longrightarrow \mathbb{Z}$ be given by $f(x) = x+1,$ for all $x \in \mathbb{Z}.$ Further, let $X = \sqcup_{\alpha \in \Lambda} \{X_\alpha\}$ with $X_{\alpha} = \mathbb{Z}$ and $f_{\alpha,\beta}:X_{\alpha} \longrightarrow X_{\alpha}$ given by $f_{\alpha,\beta} = f,$ for all $\alpha, \beta \in \Lambda$ and $\alpha \neq \beta.$
	\end{example}
	
\section{Rack and quandle homology}\label{section 2}

	In this section, we start by recalling the definitions of rack and quandle homology. The first ideas of rack homology dates back to April 2, 1990 in a letter written by R. Fenn to C. Rourke \cite{FR}. For the history of quandle homology, see \cite{Car}. Let $C_n^R = \mathbb{Z}X^n$ for a shelf $(X,*)$ and $n>0.$ Further, let $C_0^R = 0.$ Let $\partial_n:C_n \longrightarrow C_{n-1}$ with $\partial_n$ given as follows for $(x_1,x_2,...x_n) \in X^n$: {\small $$\partial_n(x_1,x_2,...,x_n) = \sum_{i=2}^{n}(-1)^i\{(x_1,x_2,...,x_{i-1},x_{i+1},x_{i+2},...,x_n)-(x_1*x_i,x_2*x_i,...,x_{i-1}*x_i,x_{i+1},x_{i+2},...,x_n)\}.$$}
	Then, $\partial_n \cdot \partial_{n+1} = 0,$ so that the $n^{th}$ {\color{blue}rack homology group} is given by: $$H_n^R(X) =\frac{ker(\partial_{n})}{im(\partial_{n+1})}.$$
	
	Let $C_n^D$ be the subset of $C_n^R$ generated by $n$-tuples $(x_1,x_2,...,x_n)$ with $x_i=x_{i+1}$ for some $i \in \{1,...,n-1\}.$ If $X$ is a quandle, then $C_n^D$ is a sub-complex of $C_n^R.$ Let $C_n^Q:=C_n^R/C_n^D$ and $\partial_{n}^Q:=\partial_{n}^R$ with the induced homomorphism. Then, the $n^{th}$ {\color{blue}quandle homology group} is given by: $$H_n^Q(X) =\frac{ker(\partial_{n})}{im(\partial_{n+1})}.$$
	
	Following is a short survey of main results proven in rack homology.
	
	\begin{theorem}[Litherland-Nelson, \cite{LN}]
		For a quandle $(X,*)$, the long exact sequence of quandle homology: $$\longrightarrow H_{n+1}^Q(X) \longrightarrow H_n^D(X) \longrightarrow H_n^R(X) \longrightarrow H_n^Q(X) \longrightarrow H_{n-1}^D(X) \longrightarrow$$ splits into short exact sequences: $$0 \longrightarrow H_n^D(X) \longrightarrow H_n^R(X) \longrightarrow H_n^Q(X) \longrightarrow 0.$$ In particular, if $H_n^D(X)$ denotes the $n^{th}$ degenerate homology group, then $$H_n^R(X) = H_n^Q(X) \oplus H_n^D(X).$$
	\end{theorem}

	Furthermore, Litherland and Nelson gave an explicit formula when $n = 2,$ and $n=3.$
	
	\begin{theorem}[Litherland-Nelson, \cite{LN}]
		 For any quandle $(X,*)$, we have:
		$$ H_2^R(X) = H_2^Q(X) \oplus \mathbb{Z}\mathcal{O}, \ and$$
		$$ H_3^R(X) = H_3^Q(X) \oplus H_2^Q(X) \oplus \mathbb{Z} \mathcal{O}^2.$$
	\end{theorem}

	In \cite{PP}, this formula is generalized to a general K\"unneth type formula allowing computation of degenerate homology $H_n^D(X)$ (hence also the rack homology $H_n^R(X)$)
	from quandle homology $H_k^Q(X),$ for $k\leq n.$

	\begin{theorem}[Przytycki-Putyra, \cite{PP}]
		
		For any quandle $(X,*)$, we have:
		$$ H_n^R(X) =  H_{n}^Q(X) \oplus H_{n-1}^Q(X) \oplus \bigoplus_{p+q=n-1;p,q\geq 1} H^R_p(X)\otimes H^Q_q(X) \oplus \bigoplus_{p+q=n-2;p,q\geq 2} Tor(H_p^R(X),H_q^Q(X)).$$
		
	\end{theorem}

	In particular, 
	\begin{itemize}
		\item[\Coffeecup]{$ H_4^R(X) =
			H_4^Q(X) \oplus H_3^Q(X) \oplus ( \mathbb{Z} \mathcal{O}\otimes H_2^Q(X))^2 \oplus \mathbb{Z}\mathcal{O}^2.$}
		\item[\Coffeecup]{$ H_5^R(X) = H_5^Q(X) \oplus H_4^Q(X) \oplus (H_3^R(X) \otimes H_1^Q(X)) \oplus (H_2^R(X) \otimes H_2^Q(X)) \oplus H_1^R(X) \otimes H_3^Q(X).$}	\end{itemize}
	
	\begin{theorem}[Etingof-Grana, \cite{EG}]
		Let $k$ be the number of orbits of a finite rack $(X,*).$ Then,
		\begin{enumerate}
			\item[\Coffeecup]{$rank(H_n^R(X)) = k^n,$ and}
			\item[\Coffeecup]{$rank(H_n^Q(X)) = k(k-1)^{n-1},$ if additionally $(X,*)$ is a quandle.}
		\end{enumerate}
	\end{theorem}

	\begin{theorem}[Niebrzydowski-Przytycki, \cite{NP1}]
		$H_n^R(X)$ for $n \geq 3$ contains $\mathbb{Z}_p$ torsion where $X$ is the dihedral quandle of order $p$.
	\end{theorem}

	\begin{theorem}[Przytycki-Yang, \cite{PY}]
		Let $(Q,*)$ be a finite quasigroup quandle. Then the torsion subgroup of $H_n^R(Q)$ is annihilated by $|Q|.$
	\end{theorem}

	The above result was proposed as a conjecture in \cite{NP2} and special cases of the above theorem was proven in \cite{Cla,Nos1}.
		
\subsection{The second rack homology group}

In this subsection, we concentrate on the quandle $GQ(o_1 \mid o_2 \mid \cdots \mid o_k).$ In addition, we also allow $k$ to be infinite and orbits to be $\mathbb{Z}.$ When an orbit $O_i$ is $\mathbb{Z},$ we denote its elements by $\{...,a^i_{p-1},a^i_p,a^i_{p+1},... \},$ for $p \in \mathbb{Z}$ and the function $f_{i,j}$ for all $0 \leq j \neq i \leq k,$ is given by: $f_{i,j}:O_i \longrightarrow O_i,$ with $f_{i,j}(a^i_p) = a^i_{p+1},$ for all $p \in \mathbb{Z}.$ When an orbit $O_i$ is finite, we denote it by $\{a^i_0,a^i_1,...,a^i_{o_i-1}\},$ and the function is given by: $f_{i,j}:O_i \longrightarrow O_i,$ with $f_{i,j}(a^i_p) = a^i_{p+1},$ for all $p \in \mathbb{Z}_{o_i}.$ 

Recall that, by definition, all the orbits $O_i$ are right action orbits of a quandle $(X,*)$ and the chain complex $C_*^R(X)$ splits: 
$$C_*^R(X)= C^{O_1}_*(X) \oplus C^{O_2}_*(X) \oplus \cdots \oplus C^{O_{k}}_*(X)$$ where $C^{O_i}_*(X)$ is a subchain complex with basis $O_i\times X^{n-1}$. Thus, it suffices to work with only $C^{O_i}_*(X).$ Let $e_p^i = a_p^i - a_{p-1}^i,$ for all $p \in \mathbb{Z}_{o_i}$ and $E_i = \{e_p^i\}^{}_{0<p<o_i}.$ Notice that $\{a_0^i\}\cup E_i$ is a basis of $\mathbb{Z}O_i$. For a basis of $\mathbb{Z}(O_i \times O_j)$ we also can consider: $(\{a_0^i\}\times \{a_0^j\}) \sqcup (E_i \times E_j) \sqcup (\{a_0^i\} \times E_j) \sqcup (E_i \times \{a_0^j\}).$ For simplicity, we write $C^i_*(X)$ for $C^{O_i}_*(X),$ $a_p$ for $a_p^1 \in O_1$ and $e_p$ for $e_{p}^1.$ 

We continue with our main theorem:

\begin{theorem}
	
	Let $(X,*) = GQ(o_1 \mid o_2 \mid \cdots \mid o_k).$ As before, let $O$ denote the set of orbits of $(X,*).$ Then, 
	\begin{enumerate} 
		\item{$H_1^R(X)= \mathbb{Z}\mathcal{O}.$ In particular, when $k$ is finite, $H_1^R(X)= \mathbb{Z}^k$.}
		 \item{For $k=2$,  $1 \leq o_1,o_2 < \infty$ we have $H_2^R(X)= \mathbb{Z}^4\oplus \mathbb{Z}^2_{gcd(o_1,o_2)}$.}
		\item{For $k>3,$ $H^1_2(X)= \mathbb{Z}{\mathcal O} \oplus \mathbb{Z}_{gcd(2, o_1,o_2,...,o_k)}$ if each $X_j$ is finite.}
	\end{enumerate}

	The case when $X_j$ is infinite ($X_j=\mathbb{Z}$) is explained in Remark \ref{rem29} and Lemma \ref{210}. We stress here that the torsion part is given by the same formulae but the free part does not necessarily conform to the formula valid for finite quandles.
\end{theorem}

The formula of (1) in the above theorem holds for any rack and follows directly from the fact that $\partial_2(x_1,x_2)= (x_1 - x_1*x_2)$. For finite $o_1,$ $\partial_2(C^1_2(X))$ is freely generated by
$(o_1-1)$ elements: $\partial_2(a_0,a_0^2)=e_1$,..., $\partial_2(a_{o_1-2},a_0^2)=e_{o_1-1}$.
From this, it follows that $H^1_1(X)= \mathbb{Z}$ with the generator $[a_0]\equiv [a_1]\equiv [a_2]\equiv \cdots \equiv [a_{o_1-1}]$.
We can change the basis $O_1\times X$ of $C^1_2(X)$ to a new basis  composed of two parts: the basis of $ker( \partial_2)$ and the basis of a subspace isomorphic to
$ im(\partial_2)$. The first part is composed of $O_1\times O_1$ and $\{(a_i,e_j^2)\}$, where $0\leq i <o_1, \  0<j<o_2$, and
$(\sum_{i=0}^{o_1-1}a_i,a_0^2).$ The second part is composed of $\{(a_i,a_0^2)\}_{0\leq i \leq o_1-2}$.

We start the proof of the formula for the second homology by considering the case: $k = 2$ and $0<o_1,o_2<\infty.$ For simplicity, we work with $C^1_*(X)$ and its second rack homology group $H_2^1(X).$

\begin{lemma}\label{Lemma 2.2}
	
	$ H^1_2(X) = \mathbb{Z}^2 \oplus \mathbb{Z}_{gcd(o_1,o_2)},$ where the free part is generated by the class $[(a_0,a_0)]$ and $[\sum_{i=0}^{o_1-1}(a_i,a_0^2)]$ 
	and the torsion part by $[(e_1,a_0)]\equiv [(a_0,e_1^2)]$. 
\end{lemma}

\begin{proof} 

	The kernel $\partial_2:C_2^1(X) \to C_1^1(X)$ has a basis:
	$$(O_1 \times O_1) \sqcup (E_1\times E_2) \sqcup (\{a_0\}\times E_2) \sqcup [\sum_{i=0}^{o_1-1}(a_i,a_0^2)].$$ 
	
	Now we analyze the image of $\partial_3^R: C_3^1(X) \to C_2^1(X)$. For $x,y,z \in X,$ $$\partial_3^R(x,y,z)= (x,z)- (x,y) - (x*y,z)+ (x*z,y*z),$$ that is,
	\begin{equation*} 
	\partial_3^R(x,y,z)=
	\begin{cases}
	0, & \text{if $y, z \in O_1$} \\
	-(x,y) + (f(x),f(y))  & \text{if $y\in O_1, z\in O_2$}\\
	(x,z)- (x,y) - (f(x),z) + (x,f(y)) & \text{if $y\in O_2, z\in O_1$}\\ 
	(x,z)- (x,y) - (f(x),z) + (f(x),y) & \text{if $y,z \in O_2$}
	\end{cases}
	\end{equation*}
	
	In the above equation, if $x = a_p^i,$ $f(x)$ denotes the element $a_{p+1}^i.$ Similarly, $f(y)$ is defined. Therefore, the relations obtained in $ker(\partial_2)$ from each case are:
	\begin{enumerate}
		\item[$(I)$]{$x,y\in O_1, z \in O_2$ gives $(x,y) \equiv (f(x),f(y)),$ or equivalently: $(a_i,a_j)\equiv (a_{i+1},a_{j+1})$.}
		\item[$(II)$]{$x,z\in O_1, y\in O_2$ gives    $(f(x)-x,z)\equiv (x,f(y)-y)$. We break this relation into 3 independent classes:
			\begin{enumerate}
				\item[$(a)$]{$[(e_i,e_j)]\equiv 0.$ This relation is obtained by considering different elements $z,z' \in O_1$ to obtain:
					$(f (x) - x, z' - z) \equiv 0,$ which is equivalent to $[(e_i,e_j)]\equiv 0$.}
				\item[$(b)$]{$[(a_0,e_i^2-e_1^2)]\equiv 0$. This relation is obtained by considering different elements $y,y' \in O_2.$}
				\item[$(c)$]{$[(e_1,a_0)]\equiv [(a_0,e_1^2)].$}
			\end{enumerate}}
		\item[$(III)$]{$[(e_i,e_j^2)] \equiv 0$.}
	\end{enumerate}

	We will now analyze $H_2^1(X)$ step by step.
	
	\begin{enumerate} 
	
	\item { $$\frac{\mathbb{Z}(O_1\times O_1)}{<I,IIa>}= \mathbb{Z} \oplus \mathbb{Z}_{o_1}$$ generated by $(a_0,a_0)$ and $(e_1,a_0)$ with $o_1(e_1,a_0)\equiv 0$.\\
	To show this, let us consider the following basis of  $\mathbb{Z}(O_1 \times O_1)$: $\{(a_0,a_0),(e_i,e_j), (a_0,e_j),(e_i,a_0) \mid 0 < i,j < o_1 \}.$ In $\mathbb{Z}(O_1\times O_1)/(IIa),$ the basis reduces to $\{(a_0,a_0),(a_0,e_j),(e_i,a_0) \mid 0<i,j<o_1\}.$ Relations $(I)$ modulo $(IIa)$ are now of the form:
	$$0\equiv (a_{i+1},a_{j+1})- (a_i,a_j)$$
	$$=(a_0+e_1+ \cdots +e_{i+1},a_0+e_1+ \cdots +e_{j+1})-(a_0+e_1+ \cdots +e_{i},a_0+e_1+ \cdots +e_{j})$$ 
	$$\equiv (a_0,e_{j+1})+(e_{i+1},a_0).$$
	The relation holds for all $0 \leq i,j \leq o_1-1$, where $a_{o_1}=a_0$ and $e_{o_1}=a_0 - a_{o_1-1}$. Using equations for $i,j < o_1-2,$ we obtain:
	$$(a_0,e_1)\equiv (a_0,e_2)\equiv \cdots \equiv (a_0,e_{o_1-1})\equiv -(e_1,a_0)\equiv -(e_1,a_0)\equiv \cdots \equiv -(e_{o_1-1},a_0).$$
	
	If we consider the relation $(e_{o_1},a_0)\equiv (e_1,a_0)$ and use the fact that $e_{o_1}= - (a_0-a_{o_1-1}) = -((a_1-a_0) + (a_2-a_1)+ \cdots +(a_{o_1-1}-a_{o_1-2}))
	= -(e_1+e_2+ \cdots +e_{o_1-1})$ we have, $(e_{o_1},a_0)\equiv (e_1,a_0)$ which is equivalent to $o_1(e_1,a_0)\equiv 0$. Other relations involving $e_{o_1}$ do not bring any new relations.}
	
	\item{$$\frac{(O_1\times O_2 \cap ker(\partial_2^R))}{(IIb,III)} = \mathbb{Z}\oplus \mathbb{Z}_{o_2}$$ generated by $(\sum_{i=1}^{o_1-1}a_i,b_0)$ and $(a_0,e_1^2)$ 	with $o_2(a_0,e_1^2)\equiv 0$.
		
	We proceed as in part (1) considering first $$\frac{(X_1\times X_2 \cap ker \partial_2)}{(III)}$$ to get
	free group with basis $(a_0,e_j^2)$ and $(\sum_{i=0}^{o_1-1}a_i,a_0^2)$.
	Now we add relations $(IIb)$ that is $(a_0,e_j^2-e_1^2)\equiv 0$. This equation holds for all $0< j \leq o_2$ 
	so including also  $e_{o_2}^2=a_0^2-a_{o_2-1}^2=
	-(e_1^2+e_2^2+ \cdots +e_{o_2-1}^2)$. Thus equations from $(IIb)$ for $j<o_2$ give $(a_0,e_1^2)\equiv (a_0,e_2^2)\equiv \cdots \equiv (a_0,e_{o_2-1}^2)$ 
	and for $j=o_2$ we get additionally equation equivalent to $o_2(a_0,e_1^2)\equiv 0$.}

	\item {Observe that adding relation $(IIc)$ is making the tensor product of $Z_{o_1}\otimes Z_{o_2}$ and thus $$H_2^1(X) = \frac{ker (\partial_2)}{(I,II,III)} = \mathbb{Z} \oplus (\mathbb{Z}_{o_1}\otimes \mathbb{Z}_{o_2}) \oplus \mathbb{Z}.$$}
	
\end{enumerate}
\end{proof}

\begin{remark}\label{rem29}
	
	A natural question now is that what happens if $o_1$ or $o_2$ is infinite. The proof is very similar except that:
	
	\begin{enumerate}
		\item{If both $o_1$ and $o_2$ are infinite then $H_2^1(X)$ is free with ($H_2^1=\mathbb{Z}^2$) with basis classes: $(a_0,a_0)$ and $(a_0,e_1)$.}
		\item{If $o_1=\infty $ but $o_2$ is finite, then there is no free term $\sum_{i=0}^{o_1-1}a_i$  and the element $[(e_1,a_0)]$ is a free element in $\mathbb{Z}(O_1\times O_1)/(I,IIa)$ which in homology generates $\mathbb{Z}_{o_2}$, and thus $H_2^1(X)= \mathbb{Z}\oplus \mathbb{Z}_{o_2}$.}
		\item{Similarly if $o_2=\infty $ but $o_1$ is finite then $[(a_0,e_i^2)]$ is free in $$\frac{(O_1\times O_2 \cap ker \partial_2)}{(IIa,III)}.$$ But in $H_2^1(X)$ it gives torsion $Z_{o_1}.$ Therefore $H_2^1(X)= \mathbb{Z} \oplus \mathbb{Z}_{o_1}.$}
	\end{enumerate}

	In cases (2) and (3), $tor(H_2^R(X)= tor(H_2^Q(X)= \mathbb{Z}_{gcd(o_1,o_2)}$ while in the case (1) $tor(H_2^R(X)= tor(H_2^Q(X)= 0$. Further, if $X$ is infinite then the results in \cite{EG,LN} does not necessarily hold. If $o_1$ is finite but $o_2$ is infinite then $free (H^R_2(X))= \mathbb{Z}^2$ and $free(H^Q_2(X)=0).$
	
	Observe that for the degenerate part of the second rack homology group, we have immediately that $H^{D}_2(X)=\mathbb{Z}^2$ generated by the classes: $[(a_0,a_0)] \equiv [(a_i,a_i)],$ for $0<i<o_1$ and $[(a_0^2,a_0^2)] \equiv [(a_i^2,a_i^2)],$ for $0<i<o_2$ so that $H^D_2(X)= \mathbb{Z}{O}$ and it holds for finite or infinite $X$. Hence, using this and the splitting theorem in \cite{LN}, we know the second quandle homology group as well.

\end{remark}

We will now prove the theorem for $k\geq 3.$ To deal with infinite quandles we define ${O}_f$ to be the set of orbits with finite $O_j.$ 

\begin{lemma}\label{210} \
	\begin{enumerate} 
		\item{$H^1_1(X)= \mathbb{Z}$ generated by the class $[a_0]\equiv [a_1]\equiv [a_2]\equiv \cdots \equiv [a_{o_1-1}]$.}
		\item{For $k \geq 3$ we have $H_2^1(X)= \mathbb{Z}^k \oplus \mathbb{Z}_{gcd(2, o_1,o_2,...,o_k)}$ for finite $X$ of $k$ orbits. More generally, for quandles which are infinite:
		$$tor H_2^1(X) = \mathbb{Z}_{gcd(2, o_1,o_2,...,o_k)},$$ and the free part of $ H_2^1(X)$ is equal to $\mathbb{Z}O$ if $o_1$ is finite; otherwise it loses one $\mathbb{Z}$.} 
	\end{enumerate}
\end{lemma}

\begin{proof} The first part is same as in the case of $k=2.$ For the general case, $k\geq 3,$ and in addition to the relations $(I)$, $(II)$, and $(III)$,
	we have to take into account the relations when $x,y,z$ are all in different orbits that is of type $(IV)$ given by,
	$\partial(x,y,z)$ where $x\in O_1$ $y \in O_i$ ($i \neq 1$), and $z \in O_j$ ($j \neq 1,i$). We have: $$\partial_3(x,y,z)= (x,z)-(x,y) - (f(x),z) + (f(x),f(y)) = (x,z-y) - (f(x), z-f(y).$$
	
	\begin{enumerate}
		\item[$(IV)$]{$(x,z-y) \equiv (f(x),z-f(y)).$}
	\end{enumerate}
	
	If we consider the role of $y$ and $z$ exchanged, that is, if we consider $\partial_3(x,z,y)$ we get: $$(x,y-z) \equiv (f(x),y-f(z).$$ Computing $\partial_3((x,y,z)+(x,z,x))$) we get:
	$$0\equiv (f(x),z-f(y)) + (f(x),y-f(z))\equiv (f(x),z-f(z)) +(f(x),y-f(y).$$ In particular $(f(x),f(y)-y)$ does not depend on the choice of $y\in X_2$ and hence is equivalent to $(f(x),e_1^{X_2}).$ By previous calculations, $(a_0,e_1^{X_2})$ is equivalent to $(e_1,a_0)$. Similarly $(f(x),f(z)-z)$ is equivalent to $(e_1,a_0)$. Thus our new relation gives $2(e_1,a_0)\equiv 0$.
	 
	Observe relation $(IV)$ $(f(x)-x,y-z) \equiv (f(x),f(z)-z).$ We already showed that $(f(x)-x,z-y)$ is equivalent to $(f(x),f(y)-y)$ and then  to $(a_0,e_1),$ which is annihilated by $2$ (and $o_1,...,o_k$). Thus, with our reductions, the free part is generated by $[(a_0,a_0)]$, $[(\sum_{i=0}^{o_i}(a_i,a_0^{2})]$, and $[(a_0,a_0^j-a_0^2)]$ for $j>2,$ while the finite part is generated by $(a_0,e_1)$ of order $gcd(2,o_1,...,o_k)$.
	
	The fact that there are no more relations essentially follows from our proof. The proof for the infinite case follows similarly. If $O_1$ is infinite we lose one $\mathbb{Z}$ in the free part of $H_2^1(X).$ The infinite sum $\sum_{i=0}^{o_1-1}(a_i,a_0^{2})$ is no more a chain.
	
\end{proof}

\subsection{Quandle co-cycle invariants and Abelian extensions of graphic quandles}
	
	In this subsection, we recall the initial motivation for quandle homology. We also discuss Abelian extensions of quandles.
	
	Let $D$ be a link diagram and $(X,*),$ a fixed quandle. A {\color{blue} {\bf quandle coloring}} of $D$ is a function $\mathfrak{C}:arc(D) \longrightarrow X,$ where $arc(D)$ denotes the set of arcs of $D$ and at each crossing of $D,$ the coloring rule described in Figure \ref{coloring} holds.
		
	\begin{figure}\label{coloring}
		\centering
		 \includegraphics[scale=0.5]{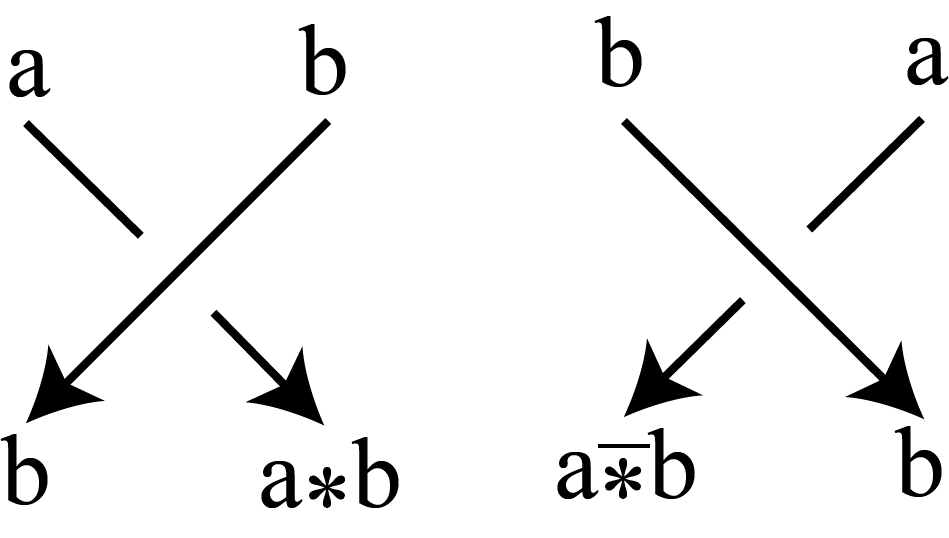}
		 \caption{Coloring rules of positive (on the left) and negative crossings}
	\end{figure}

	2-cocycles in quandle cohomology theory lead to, on one hand to 2-cocycle invariants for classical knots and links and on the other hand to extension of quandles.\footnote{Laver tables are examples of shelves which are not racks. The 2 and 3-cocycles of Laver tables (\cite{Lav}) were computed by P. Dehornoy and V. Lebed (\cite{DL}).} The idea of the 2-cocycle invariant is summarized in the following theorem.

	\begin{theorem}[Carter et al., \cite{CJKLS}]
	
	Let $\mathfrak{C}$ be a coloring of an oriented link diagram $L$ using a quandle $(X,*)$. Let $\phi$ be a 2-cocycle of $(X,*)$ with coefficients in an Abelian group $(A,+).$ Associate to each crossing $c$ of $L,$ a (Boltzmann) weight $\phi(a,b)^{\epsilon},$ where $a$ and $b$ are the tail of the arcs as in Figure \ref{coloring} and $\epsilon$ is $+1$ if the crossing $c$ is positive and $-1$ if it is negative. Then, the {\color{blue}{\bf 2-cocycle invariant}} is given by: $$\sum_{\mathfrak{C}} \prod_{c} \phi(x,y)^{\epsilon}$$ is a link invariant, where the sum is taken over all the colorings of $D$ with the quandle $(X,*)$.
	
	\end{theorem}

	The 2-cocycle of group cohomology allows a group to be extended to obtain a larger group. The analogous idea works for quandle homology and quandles.
	
	\begin{theorem}[Carter et al., \cite{CJKLS}]
			Let $(X,*)$ be a quandle and $\phi$, a 2-cocycle with coefficients in an Abelian group $(A,+)$. Then, $A\times X$ can be given a quandle structure with $\circ: (A \times X) \times (A \times X) \longrightarrow A \times X$ given by, $$(a,x) \circ (b,y) = (a + \phi(x,y) , x*y),$$ for all $a,b \in A$ and $x,y \in X.$ 
	\end{theorem}

	We illustrate this by the following example of extending graphic quandles.
	
	\begin{proposition}
		
	Let $GQ(o_1\ | \ o_2 \ | \ ... \ | \ o_k)$ be the graphic quandle with $n$ orbits of size $o_1,o_2,...,o_k$ respectively. Define the 2-cocycle $\phi: GQ(o_1\ | \ o_2 \ | \ ... \ | \ o_k) \times GQ(o_1\ | \ o_2 \ | \ ... \ | \ o_k) \longrightarrow \mathbb{Z}_n,$ by $\phi(a,b) = 0$ if $a$ and $b$ are in the same orbits and $\phi(a,b) = 1$ otherwise. Then, the Abelian 2-cocycle extension of $GQ(o_1\ | \ o_2 \ | \ ... \ | \ o_k)$ by $\mathbb{Z}_n$ is isomorphic to $GQ(n \cdot o_1\ | \ n \cdot o_2 \ | \ ... \ | \ n \cdot o_k).$ In particular, if $o_1,o_2,...o_k = 1,$ we extend the trivial quandle of $k$ elements by $\mathbb{Z}_n$ and obtain $GQ(n \ | \ n \ | \ ... \ | \ n)$ with possible non-trivial torsion.
	
\end{proposition}
	
	\section{Odds and ends}
	
	In this paper, we concentrated on a specific family of graphic quandles. Firstly, the choice of functions from the appropriate permutation groups decide a lot about the structure of the quandle. The only condition these elements have to satisfy for a given orbit is that they should commute pairwise. In this article, we restricted our choices to same cyclic permutations for every orbit of the quandle (mostly, because we wanted to understand completely the second rack homology of these quandles). So far, we have been unable to construct a finite graphic quandle which is not part of the family introduced in Example \ref{graphic_quandle_example}. We checked that there is no such example up to order six.
	
	Graphic shelves remain unexplored. In particular, as we focused on quandles, we did not consider the one term homology of graphic shelves and graphic spindles which are not racks. Additionally, we did not consider multi term homology for distributive sets.
	
	The main example of graphic quandles we constructed in somewhat similar to the construction of $f$-block spindles introduced in \cite{CPP}. In particular, these spindles are proven to be very rich from one term homology point of view. It is a natural to ask if there is some similarity between these.  
	
	The main construction used in this paper for graphic quandles can be generalized to construct biracks and biquandles as well. However, in this paper since we concentrate on the rack and quandle homology of self distributive algebraic structures, we do not discuss this notions further.
				
	Based on our preliminary computational data, we propose the following conjectures.
	{\scriptsize
	\begin{table}[h]
		\caption{The graphic quandles $GQ( Id ,(0 \ 1 \ 2 \ 3),(0 \ 2)(1 \ 3) \mid (4 \ 6)(5 \ 7),Id,(4 \ 5 \ 6 \ 7) \mid (8 \ 9 \ 10 \ 11),(8 \ 10)(9 \ 11),Id)$ (on the left) and $GQ( Id,(0 \ 1 \ 2 \ 3),(0 \ 2)(1 \ 3) \mid (4 \ 5 \ 6 \ 7),Id,(4 \ 6)(5 \ 7) \mid (8 \ 9 \ 10 \ 11),(8 \ 10)(9 \ 11), Id).$}
		\label{GQ}
		\begin{minipage}{0.5\textwidth}
			\centering
			\begin{tabular}{c || c c c c | c c c c | c c c c}
				*&0&1&2&3&4&5&6&7&8&9&10&11 \\
				\hline \hline
				0&0&0&0&0&1&1&1&1&2&2&2&2\\
				1&1&1&1&1&2&2&2&2&3&3&3&3\\
				2&2&2&2&2&3&3&3&3&0&0&0&0\\
				3&3&3&3&3&0&0&0&0&1&1&1&1\\	
				\hline
				4&6&6&6&6&4&4&4&4&5&5&5&5\\
				5&7&7&7&7&5&5&5&5&6&6&6&6\\
				6&4&4&4&4&6&6&6&6&7&7&7&7\\
				7&5&5&5&5&7&7&7&7&4&4&4&4\\
				\hline
				8&9&9&9&9&10&10&10&10&8&8&8&8\\
				9&10&10&10&10&11&11&11&11&9&9&9&9\\
				10&11&11&11&11&8&8&8&8&10&10&10&10\\
				11&8&8&8&8&9&9&9&9&11&11&11&11\\
			\end{tabular}
		\end{minipage}
		\hfillx
		\begin{minipage}{0.5\textwidth}
			\centering
			\begin{tabular}{c || c c c c | c c c c | c c c c}
				*&0&1&2&3&4&5&6&7&8&9&10&11 \\
				\hline \hline
				0&0&0&0&0&1&1&1&1&2&2&2&2\\
				1&1&1&1&1&2&2&2&2&3&3&3&3\\
				2&2&2&2&2&3&3&3&3&0&0&0&0\\
				3&3&3&3&3&0&0&0&0&1&1&1&1\\	
				\hline
				4&5&5&5&5&4&4&4&4&6&6&6&6\\
				5&6&6&6&6&5&5&5&5&7&7&7&7\\
				6&7&7&7&7&6&6&6&6&4&4&4&4\\
				7&4&4&4&4&7&7&7&7&5&5&5&5\\
				\hline
				8&9&9&9&9&10&10&10&10&8&8&8&8\\
				9&10&10&10&10&11&11&11&11&9&9&9&9\\
				10&11&11&11&11&8&8&8&8&10&10&10&10\\
				11&8&8&8&8&9&9&9&9&11&11&11&11\\
			\end{tabular}
		\end{minipage}
	\end{table}
	}
	\begin{conjecture}
		Let $GQ(o_1\ | \ o_2 \ | \ ... \ | \ o_k)$ be the graphic quandle with $k$ orbits of size $o_1,o_2,...,o_k$ respectively, with at least on orbit having more than two elements. Let $d = gcd(o_1,o_2,...,o_k).$ Then, if $d \neq 1,$ $\mathbb{Z}_d \subseteq H_n^Q(GQ(o_1\ | \ o_2 \ | \ ... \ | \ o_k)),$ when $n \geq k$ and $\mathbb{Z}_d \nsubseteq H_{n-1}^Q(GQ(o_1\ | \ o_2 \ | \ ... \ | \ o_k)),$ when $n < k.$
	\end{conjecture}

	\begin{conjecture}
		Let $GQ(o \mid o)$ be the graphic quandle with $2o$ elements divided in two orbits of equal size. Then, $$tor H^Q_n(GQ(o \mid o))= \mathbb{Z}_o^{c_n}, $$ where $c_n$ is given as follows for $n \in \mathbb{Z}^{+}$. $$c_0  = 0 = c_1, \ c_{2n}=2c_{2n-1}+2, \ and \ c_{2n+1} = 2c_{2n}.$$ In closed form, we have $c_{2n} = 2\frac{4^n-1}{3},$ and $c_{2n+1} = 4\frac{4^n-1}{3}.$
		
	\end{conjecture}

		We next introduce a standard notation for graphic quandles like the ones shown in Table \ref{GQ} to tabulate computational data. Let $S$ be a set with $n$ elements. Let $S_n(S)$ denote the permutation group $S_n$ with its elements denoted using the elements of $S.$ Then, by $GQ( f_{1,1},f_{1,2}, ..., f_{1,k} \mid f_{2,1},f_{2,2}, ..., f_{2,k} \mid \cdots \mid f_{k,1},f_{k,2}, ..., f_{k,k} ),$ we denote the graphic quandle with orbits $O_1,O_2,...,O_k$ having $o_1,o_2,...o_k$ elements respectively with $f_{i,j} \in S_{o_i}(O_i)$. 
		
		For example, the graphic quandle on the left hand side of Table \ref{GQ} has three orbits: $\{0,1,2,3\},\{4,5,6,7\},$ and $\{8,9,10,11\}.$ It is denoted as:
		 $$GQ( Id_{S_{o_1}(O_1)} ,(0 \ 1 \ 2 \ 3),(0 \ 2)(1 \ 3) \mid (4 \ 6)(5 \ 7),Id_{S_{o_2}(O_2)},(4 \ 5 \ 6 \ 7) \mid (8 \ 9 \ 10 \ 11),(8 \ 10)(9 \ 11),Id_{S_{o_3}(O_3)} ).$$ The graphic quandle on the right hand side of Table \ref{GQ} is denoted by: 
		 $$GQ( Id_{S_{o_1}(O_1)} ,(0 \ 1 \ 2 \ 3),(0 \ 2)(1 \ 3) \mid (4 \ 5 \ 6 \ 7),Id_{S_{o_2}(O_2)},(4 \ 6)(5 \ 7) \mid (8 \ 9 \ 10 \ 11),(8 \ 10)(9 \ 11), Id_{S_{o_3}(O_3)} ).$$ The following table consists of some graphic quandles. A `?' symbol is used when a particular entry is beyond the scope of our present computational abilities.
{\footnotesize
		\begin{table}[h]\label{data}
		
		\centering
		\caption{The finite subgroups of some graphic quandles}
		\begin{tabular}{||c|c|c|c|c||}
			\hline 
			Quandle $(X,*)$&$H_1^R$&$H_2^R$&$H_3^R$&$H_4^R$\\
			\hline
			$GQ(Id,(0 \ 1),(0 \ 1) \mid (2 \ 3),Id,(2 \ 3) \mid (4 \ 5),(4 \ 5),Id)$&1&{\color{blue}$\mathbb{Z}_2^3$}&{\color{blue}$\mathbb{Z}_2^{15}$}&{\color{blue}$\mathbb{Z}_2^{75}$}\\
			\hline
			$GQ( Id,(0 \ 1 \ 2),(0 \ 1 \ 2) \mid (3 \ 4 \ 5),Id,(3 \ 4 \ 5) \mid (6 \ 7 \ 8),(6 \ 7 \ 8),Id)$&1&1&{\color{blue}$\mathbb{Z}_3^3$}&?\\
			\hline
			$GQ(Id,(0 \ 1),(0 \ 1),(0 \ 1) \mid (2 \ 3),Id,(2 \ 3),(2 \ 3) \mid (4 \ 5),(4 \ 5),Id,(4 \ 5) \mid (6 \ 7),(6 \ 7),(6 \ 7),Id)$&1&1&1&{\color{blue}$\mathbb{Z}_2^4$}\\
			\hline
			$GQ( Id ,(0 \ 1 \ 2 \ 3),(0 \ 2)(1 \ 3) \mid (4 \ 6)(5 \ 7),Id,(4 \ 5 \ 6 \ 7) \mid (8 \ 9 \ 10 \ 11),(8 \ 10)(9 \ 11),Id )$&1&1&{\color{blue} $\mathbb{Z}_4^3$}&?\\
			\hline
			$GQ( Id,(0 \ 1 \ 2 \ 3),(0 \ 2)(1 \ 3) \mid (4 \ 5 \ 6 \ 7),Id,(4 \ 6)(5 \ 7) \mid (8 \ 9 \ 10 \ 11),(8 \ 10)(9 \ 11),Id )$&1&{\color{blue} $\mathbb{Z}_2^3$}&{\color{blue} $\mathbb{Z}_2^6$}$\oplus${\color{blue} $\mathbb{Z}_4^3$}&?\\
			\hline
		\end{tabular}
	\end{table}
}

	Note the presence of $\mathbb{Z}_3$ torsion in $H_3^R(X)$ of the graphic quandle in the second row of the above table. We checked that the torsion subgroup $\mathbb{Z}_i$ is present in $H_3^R(X)$ of graphic quandles having three orbits of equal size $i,$ with the binary operation given similarly for $3<i<7$.
	
	\section*{Acknowledgements}
	
		The first author would like to thank the organizers of the `Fourth Mile High Conference on Nonassociative Mathematics' that sowed the seeds for this article. The second author was partially supported by the Simons Collaboration Grant for Mathematicians-316446 and the CCAS Dean's Research Chair award.
	
	\bibliographystyle{plain}

\end{document}